\UseRawInputEncoding
\documentclass[12pt]{article}
\usepackage[margin=1in]{geometry}
\usepackage{amsmath, amssymb, amsthm}
\usepackage{xcolor}
\usepackage{graphicx}

\newtheorem{theorem}{Theorem}
\newtheorem{definition}[theorem]{Definition}
\newtheorem{remark}[theorem]{Remark}

\newtheorem{corollary}[theorem]{Corollary}

\newtheorem{lemma}[theorem]{Lemma}

\newcommand{\SOS}{\operatorname{SOS}}
\newcommand{\BSR}{\operatorname{BSR}}

\title{Exact and Asymptotic Values for Weak Limited Augmented Zarankiewicz Numbers in the $m\times 3$ Case}
\author{Liqun Qi\footnote{Jiangsu Provincial Scientific Research Center of Applied Mathematics, Nanjing 211189, China.
Department of Applied Mathematics, The Hong Kong Polytechnic University, Hung Hom, Kowloon, Hong Kong.
({\tt maqilq@polyu.edu.hk})}
\and
Johan L\"{o}fberg\footnote{Division of Automatic Control, Link\"{o}ping University, Sweden. ({\tt johan.lofberg@liu.se})}
\and
Yannan Chen\footnote{School of Mathematical Sciences, South China Normal University, Guangzhou 510631, China ({\tt ynchen@scnu.edu.cn}).}
}
\date{\today}

\begin{document}

\maketitle

\begin{abstract}
We determine the exact weak limited augmented Zarankiewicz numbers $z_{wL}(m,3)$ for all $m\ge 3$:
\[
z_{wL}(m,3)=
\begin{cases}
m+3+\left\lceil \dfrac{m}{2}\right\rceil+1, & 9\le m\le 15,\\[2mm]
m+3+\left\lfloor \dfrac{2m-4}{3}\right\rfloor, & m\ge 16,
\end{cases}
\]
with $z_{wL}(3,3)=6$, $z_{wL}(4,3)=8$, and $z_{wL}(m,3)=2m$ for $5\le m\le 9$. In particular,
\[
\lim_{m\to\infty} \frac{z_{wL}(m,3)}{m} = \frac{5}{3}.
\]
The proof is fully analytic, relying on a uniform base classification, two constructive lower-bound families (staircase and $5m/3$), and a sharp upper-bound argument based on a peeling lemma and the analysis of two W2-sensitive boundary cases. Numerical MILP computations were used only as proof-mining tools to identify the structural lemmas; the final theorem is unconditional. We also extend the known range of the original limited numbers $z_L(m,3)$ through $m=13$, where the gap to $z_{wL}(m,3)$ is only 2 or 3.
\end{abstract}


\section{Introduction}

The weak limited augmented Zarankiewicz number $z_{wL}(m,n)$ was introduced in \cite{QCX26b} as a relaxation of the original augmented framework \cite{QCX26,QCX26a}  for the Zarankiewicz problem \cite{Za51,Re58,Gu69}. The weak conditions (S), (W1), (W2), (W2'), and (W3) guarantee irreducibility of the associated doubly simple biquadratic forms, with
\[
\SOS(P_G) = |E_1| + |E_2|.
\]
Consequently,
\[
\BSR(m,n) \ge z_{wL}(m,n).
\]

In this paper, we focus on the $m\times 3$ case. The classical Zarankiewicz numbers $z(m,3)$ are known:
\[
z(m,3)=m+3 \quad (m\ge 3),
\]
with the exception $z(2,3)=4\neq 5$. The identity $z(m,3)=m+3$ follows from the classical theorem of \v{C}ul\'{\i}k \cite{Culik56}; see Chen, Horsley, and Mammoliti \cite{CHM24} for a modern reference.

The known results for the limited augmented Zarankiewicz numbers $z_L(m,3)$ from \cite{QCX26} and \cite{QCX26a} are:
\[
z_L(2,3)=4,\quad z_L(3,3)=6,\quad z_L(4,3)=8,\quad z_L(5,3)=9,\quad z_L(6,3)=11.
\]

In \cite{QCX26b}, the lower bounds
\[
z_{wL}(5,3)\ge 10,\qquad z_{wL}(6,3)\ge 12
\]
were established.

In this paper, we present a complete analytic solution for $z_{wL}(m,3)$. Our main results are:

\begin{enumerate}
    \item \textbf{Exact values for $3\le m\le 15$:}
    \[
    z_{wL}(3,3)=6,\quad z_{wL}(4,3)=8,\quad z_{wL}(m,3)=2m\quad (5\le m\le 9),
    \]
    \[
    z_{wL}(10,3)=19,\quad z_{wL}(11,3)=21,\quad z_{wL}(12,3)=22,\quad z_{wL}(13,3)=24,
    \]
    \[
    z_{wL}(14,3)=25,\quad z_{wL}(15,3)=27.
    \]

    \item \textbf{Exact formula for all $m\ge 16$:}
    \[
    z_{wL}(m,3)=m+3+\left\lfloor \frac{2m-4}{3}\right\rfloor.
    \]

    \item \textbf{Unconditional asymptotic slope:}
    \[
    \lim_{m\to\infty} \frac{z_{wL}(m,3)}{m} = \frac{5}{3}.
    \]

    \item \textbf{A small gap to the original limited numbers:}
    For $7\le m\le 13$, we have
    \[
    z_{wL}(m,3)-z_L(m,3)\in\{2,3\},
    \]
    with gap sequence $2,2,3,2,3,2,3$. In buffer-cell terms, the strong framework requires 4 or 6 additional buffer cells.
\end{enumerate}

The paper is organized as follows. In Section 2, we recall the necessary definitions and prove the uniform base classification for $z(m,3)$. In Section 3, we present the finite exact values for $3\le m\le 15$. In Section 4, we develop general lower bounds, including lifting lemmas and the two constructive families. In Section 5, we prove the structural upper bounds, culminating in the peeling lemma and the sharpened nonpeelable core estimate. In Section 6, we analyze the two W2-sensitive boundary cases and complete the induction. In Section 7, we present the final exact formula. In Section 8, we discuss the small gap to $z_L(m,3)$ and the computational proof-mining history.

\section{Preliminaries}

We recall the definition of weak admissible augmented bipartite graphs from \cite{QCX26b}. First, we establish some notation.

A \emph{1-edge} is an occupied cell $(r,c)$. A \emph{2-edge} is a pair of distinct free cells $(i,j)$ and $(k,l)$; it is \emph{row-degenerate} if $i=k$ and \emph{column-degenerate} if $j=l$. It is \emph{nondegenerate} if it is neither row- nor column-degenerate.

For a nondegenerate 2-edge $e=(i,j;k,l)$, its two \emph{opposite cells} are $(i,l)$ and $(k,j)$. For a mixed pair consisting of a 1-edge $(r,c)$ and a 2-edge with halves $(i,j)$ and $(k,l)$ that are vertex-disjoint, the \emph{opposite-cell set} is
\[
\{(r,j),(r,l),(i,c),(k,c)\},
\]
with repetitions suppressed in degenerate cases.

For two vertex-disjoint nondegenerate 2-edges we use the opposite-cell convention of \cite{QCX26b}. The \emph{dependency graph} of the 2-edges has a directed edge $e\to e'$ if both opposite cells of $e$ are occupied and one of them is a half of $e'$.

\begin{definition}
An augmented bipartite graph $G=(S,T,E_1\cup E_2)$ is \textbf{weak admissible} if it satisfies:
\begin{itemize}
    \item[(S)] No 2-edge overlaps with any 1-edge or other 2-edge on a cell.
    \item[(W1)] The 1-edge graph $E_1$ is $C_4$-free.
    \item[(W2)] The dependency graph of nondegenerate 2-edges contains no directed cycles, except that complementary 2-cycles are allowed.
    \item[(W2')] No nondegenerate 2-edge has both opposite cells as 1-edges.
    \item[(W3)] For any vertex-disjoint pair of edges where at least one is a 2-edge, at least one opposite cell is unoccupied.
\end{itemize}
\end{definition}

The \textbf{weak limited augmented Zarankiewicz number} $z_{wL}(m,n)$ is the maximum total number of edges $|E_1|+|E_2|$ for which $G$ is weak admissible and $|E_1|=z(m,n)$.

\begin{definition}[Complementary 2-edges]
Two nondegenerate 2-edges
\[
e_1=(i,j;k,l),\qquad e_2=(i,l;k,j)
\]
are called \textbf{complementary}. They share no halves and their halves are exactly each other's opposite cells.
\end{definition}

We define a \emph{buffer cell} as a free cell that remains unoccupied after all 2-edges have been placed. In an extremal base graph, there are $2m-3$ free cells. If $s$ 2-edges are added, they occupy $2s$ free cells, leaving $t=2m-3-2s$ buffer cells. Thus the number of buffer cells is always odd.

A cell is \textbf{occupied} if it is a 1-edge or a half of a 2-edge. A \textbf{free cell} is a cell not belonging to $E_1$; by condition (S), all 2-edge halves must lie in free cells.

We will use the following terminology. A graph is \emph{peelable} if it contains a removable configuration of three singleton rows of the form
\[
H - F - H
\]
where:
\begin{itemize}
    \item[---] the central row is \emph{full} ($F$): both of its free cells are occupied by halves of 2-edges;
    \item[---] the two outer rows are \emph{half-full} ($H$): each has exactly one free cell occupied by a half of a 2-edge;
    \item[---] the central row is connected by 2-edges to each of the two outer rows;
    \item[---] these three rows are isolated from the rest of the graph in the row-incidence structure: no other 2-edges touch them.
\end{itemize}
Intuitively, this is a small ``island'' in the row-incidence graph: one full row connected by two 2-edges to two half-full rows, with no other edges touching these rows. If such a configuration exists, we can \emph{peel} it off: delete the three rows and the two incident 2-edges, and the remaining graph is still a valid weak admissible limited graph, but with $m-3$ rows and $s-2$ 2-edges. This reduction is the key to the inductive upper bound in Section~6.

A graph is \emph{nonpeelable} if no such configuration exists.

The following uniform base classification is fundamental to all subsequent arguments.

\begin{theorem}\label{thm:base}
The identity $z(m,3)=m+3$ for all $m\ge 3$ follows from the classical theorem of \v{C}ul\'{\i}k \cite{Culik56}; see Chen, Horsley, and Mammoliti \cite{CHM24} for a modern reference.

Moreover, the following structural classification holds: if an $m\times 3$ $C_4$-free bipartite graph attains this bound, then exactly three rows have degree $2$, their supports are
\[
\{1,2\},\qquad \{1,3\},\qquad \{2,3\},
\]
and all remaining $m-3$ rows have degree $1$. Consequently every extremal base graph is determined, up to row and column permutations, by the three singleton multiplicities
\[
(a,b,c),\qquad a+b+c=m-3,
\]
where $a$, $b$, and $c$ denote the numbers of singleton rows with their unique 1-edge in columns $1$, $2$, and $3$, respectively.
\end{theorem}

\begin{proof}
Let the row degrees be $d_1,\ldots,d_m$.

First, note that $\binom{d_r}{2}$ counts the number of pairs of columns that appear together in row $r$. Since the graph is $C_4$-free, a given pair of columns cannot appear together in two different rows; otherwise those two rows and the two columns would form a 4-cycle. There are only $\binom{3}{2}=3$ possible column pairs, namely $\{1,2\}$, $\{1,3\}$, and $\{2,3\}$. Therefore, no matter how many rows there are, the total number of column-pair occurrences across all rows is at most 3:
\[
\sum_{r=1}^{m} \binom{d_r}{2} \le \binom{3}{2}=3.
\]

For any integer $d\ge 0$, one has the elementary inequality
\[
d \le 1 + \binom{d}{2},
\]
with equality only for $d=1$ or $d=2$. Summing this over all rows gives
\[
|E_1|=\sum_{r=1}^{m} d_r
\le \sum_{r=1}^{m}\left(1+\binom{d_r}{2}\right)
= m + \sum_{r=1}^{m}\binom{d_r}{2}
\le m+3.
\]
Thus $z(m,3)\le m+3$. The bound is attained by taking three rows with supports $\{1,2\},\{1,3\},\{2,3\}$ and making every remaining row a singleton, so $z(m,3)=m+3$.

Now suppose equality holds. Then equality must hold in both inequalities above. From $d\le 1+\binom{d}{2}$, equality is strict for $d=0$ (since $0<1$) and for $d=3$ (since $3<4$); hence no row can have degree $0$ or $3$. Every row therefore has degree $1$ or $2$.

Equality in the second inequality also forces
\[
\sum_{r=1}^{m}\binom{d_r}{2}=3.
\]
Each degree-$2$ row contributes exactly $1$ to this sum, while each degree-$1$ row contributes $0$. Hence exactly three rows have degree $2$. Their supports must be distinct, since two degree-$2$ rows with the same support would form a $C_4$. There are only three two-element subsets of $\{1,2,3\}$, so the three supports must be exactly
\[
\{1,2\},\qquad \{1,3\},\qquad \{2,3\}.
\]
The remaining $m-3$ rows are singletons. This proves the structural classification.
\end{proof}

\begin{corollary}
In every limited $m\times 3$ augmented graph with $m\ge 3$, there are exactly
\[
3m-(m+3)=2m-3
\]
free cells before the 2-edges are inserted. If $s=|E_2|$ and $t$ denotes the number of buffer cells (free cells left unoccupied), then
\[
2s+t=2m-3.
\]
In particular, $t$ is always odd.
\end{corollary}

\section{Exact Values for $3\le m\le 15$}

We begin by summarizing the exact values for small $m$.

\begin{theorem}\label{thm:small}
The following values hold:
\[
z_{wL}(3,3)=6,\qquad z_{wL}(4,3)=8,\qquad z_{wL}(5,3)=10,\qquad z_{wL}(6,3)=12,
\]
\[
z_{wL}(7,3)=14,\qquad z_{wL}(8,3)=16,\qquad z_{wL}(9,3)=18,
\]
\[
z_{wL}(10,3)=19,\quad z_{wL}(11,3)=21,\quad z_{wL}(12,3)=22,\quad z_{wL}(13,3)=24,
\]
\[
z_{wL}(14,3)=25,\quad z_{wL}(15,3)=27.
\]
\end{theorem}

\begin{proof}
For $m=3$, the extremal 1-edge graph consists of the three degree-2 rows and has six 1-edges. Its three free cells are the unique cells in each degree-2 row. Any pair of free cells forms a nondegenerate 2-edge whose two opposite cells are 1-edges, violating (W2'). Hence $s=0$ and $z_{wL}(3,3)=6$.

For $m=4$, there is one singleton row. A weak admissible construction with one row-degenerate 2-edge on its two free cells gives total edge count 8. Suppose instead that $s\ge 2$. There are only five free cells, so $s=2$. By Lemma~\ref{lem:qbound}, we have $q\le 2$; counting the four occupied free cells then forces $F=1$, $H=0$, and $q=2$. Corollary~\ref{cor:specialL} below then gives $L=0$, so the two free cells of the full singleton row belong to two crossing edges whose other halves are precisely the two occupied special cells. These two edges point to each other across the full row, forming a forbidden directed 2-cycle. Hence $s\le 1$ and $z_{wL}(4,3)=8$.

For $5\le m\le 9$, Lemma~\ref{lem:2mUB} below gives the upper bound $z_{wL}(m,3)\le 2m$. Explicit weak admissible constructions attaining $2m$ are given in Appendix~A, and Lemma~\ref{lem:seed} gives the $m=9$ construction. Thus equality holds.

For $10\le m\le 15$, Theorem~\ref{thm:staircaseUB} gives the upper bound
\[
s\le \left\lceil \frac{m}{2}\right\rceil+1,
\]
and Theorem~\ref{thm:staircaseLB} gives the matching lower bound. The explicit values follow by substitution. \end{proof}

Thus the sequence satisfies $z_{wL}(m,3)=2m$ for $5\le m\le 9$, and follows a staircase pattern for $10\le m\le 15$:
\[
z_{wL}(10,3)=19,\quad z_{wL}(11,3)=21,\quad z_{wL}(12,3)=22,\quad z_{wL}(13,3)=24,
\]
\[
z_{wL}(14,3)=25,\quad z_{wL}(15,3)=27.
\]

A third discovery emerging from our computations is the small and structured gap between the weak-limited and original-limited numbers. By extending the exhaustive search for $z_L(m,3)$ from $m\le 6$ to $m\le 13$, we obtain
\[
z_L(7,3)=12,\quad z_L(8,3)=14,\quad z_L(9,3)=15,\quad z_L(10,3)=17,
\]
\[
z_L(11,3)=18,\quad z_L(12,3)=20,\quad z_L(13,3)=21.
\]
Comparing these with the weak-limited values reveals that
\[
z_{wL}(m,3)-z_L(m,3)\in\{2,3\}
\]
for all $7\le m\le 13$, with the gap sequence
\[
2,2,3,2,3,2,3.
\]
Equivalently, since each 2-edge occupies two buffer cells, the strong framework requires 4 or 6 additional buffer cells.

The verification method for these $z_L$ values is described in Section~8.

\section{General Lower Bounds}

We now develop unconditional lower bounds for all $m$.

\begin{lemma}[One-row lift]\label{lem:onelift}
Let $G$ be any weak admissible limited $m\times 3$ graph, $m\ge 3$. Then there is a weak admissible limited $(m+1)\times 3$ graph with exactly one more edge. Consequently
\[
z_{wL}(m+1,3)\ge z_{wL}(m,3)+1.
\]
\end{lemma}

\begin{proof}
Add a new row $p$, put the single 1-edge $(p,1)$ in that row, and leave $(p,2)$ and $(p,3)$ unoccupied. The 1-edge graph remains $C_4$-free, and by Theorem~\ref{thm:base} its size has increased from $m+3$ to $(m+1)+3$. Simplicity, (W2), and (W2') are unchanged. For (W3), an old 2-edge can be vertex-disjoint from the new 1-edge only if it avoids column 1; every opposite-cell set for such a pair contains a cell in the new row and in a column used by that 2-edge, which is unoccupied. \end{proof}

\begin{lemma}[Two-row lift]\label{lem:twolift}
Let $G$ be a weak admissible limited $m\times 3$ graph with no column-degenerate 2-edge, $m\ge 3$.
Then one may add two rows and one new 2-edge so as to obtain a weak admissible limited $(m+2)\times 3$ graph, again with no column-degenerate 2-edge. The total number of edges increases by 3.
Any weak admissible limited construction with no column-degenerate 2-edge can therefore be lifted by two rows while increasing the total number of edges by three.
\end{lemma}

\begin{proof}
Add two new rows $p,q$. Add the 1-edges $(p,1)$, $(q,2)$, and the nondegenerate 2-edge $e=(p,2;q,3)$. Leave $(p,3)$ and $(q,1)$ unoccupied. The new 1-edge graph is still $C_4$-free, and its size increases by two, exactly as required by Theorem~\ref{thm:base}. Simplicity is immediate. The 2-edge $e$ is isolated in the dependency graph: its opposite cells include $(p,3)$, which is unoccupied, and $(q,1)$, which is unoccupied. Hence (W2) and (W2') hold. For (W3), the verification is by cases: if $e$ is paired with an old 1-edge, vertex-disjointness forces the 1-edge to be in column 1, and $(q,1)$ is unoccupied; if $(p,1)$ is paired with an old 2-edge, $(p,3)$ is unoccupied; if $(q,2)$ is paired with an old 2-edge, $(q,1)$ is unoccupied. No two 2-edges are vertex-disjoint. \end{proof}

We now give an explicit $9\times 3$ seed.

\begin{lemma}\label{lem:seed}
There is a weak admissible limited $9\times 3$ graph with six nondegenerate 2-edges and hence 18 edges in total. It has no column-degenerate 2-edge.
\end{lemma}

\begin{proof}
Take rows $1,2,3$ to have 1-edge supports $\{1,2\},\{1,3\},\{2,3\}$, and take rows $4,5$ as singletons in column 1, rows $6,7$ as singletons in column 2, and rows $8,9$ as singletons in column 3. Add the complementary pairs
\[
(4,2;5,3),(4,3;5,2),\quad
(6,1;7,3),(6,3;7,1),\quad
(8,1;9,2),(8,2;9,1).
\]
All twelve halves are distinct and lie in free cells. The six 2-edges form three complementary pairs. For each edge, its two opposite cells are exactly the two halves of its complementary partner, so each pair forms one allowed complementary 2-cycle in the dependency graph, and there are no dependencies between different pairs because their row sets are disjoint. Thus (W2) holds. The same observation gives (W2'). For (W3), every 2-edge uses two distinct columns, so no two 2-edges are vertex-disjoint; for mixed pairs, the vertex-disjoint 1-edge must lie in the column of the singleton rows of the 2-edge's row pair, and the relevant degree-2 row contains an unoccupied opposite cell. \end{proof}

Applying Lemma~\ref{lem:twolift} repeatedly to Lemma~\ref{lem:seed}, and Lemma~\ref{lem:onelift} when one additional row is needed, gives the following unconditional lower bound.

\begin{theorem}\label{thm:staircaseLB}
For every $m\ge 9$,
\[
z_{wL}(m,3)\ge m+3+\left(\left\lceil \frac{m}{2}\right\rceil+1\right)
= 2m-\left\lfloor \frac{m}{2}\right\rfloor+4.
\]
\end{theorem}

However, a denser infinite family improves this for large $m$.

\begin{theorem}\label{thm:5m3LB}
For every integer $k\ge 0$, there is a weak admissible limited $(3k+5)\times 3$ graph with $2k+2$ 2-edges. Consequently,
\[
z_{wL}(3k+5,3)\ge 5k+10.
\]
\end{theorem}

\begin{proof}
We construct the graph explicitly. Let the row set be
\[
S=\{P_{12},P_{13},P_{23}\}\cup\{U,V\}\cup\bigcup_{i=1}^{k}\{R_i,A_i,B_i\},
\]
so $|S|=3+2+3k=3k+5$. Let the column set be $T=\{1,2,3\}$.

\medskip
\noindent\textbf{Step 1: The 1-edges.}
Define the 1-edge set $E_1$ as follows:
\begin{itemize}
    \item In the three special rows $P_{12},P_{13},P_{23}$, place 1-edges with supports
    \[
    \{1,2\},\qquad \{1,3\},\qquad \{2,3\},
    \]
    respectively.
    \item In the two rows $U,V$, place singleton 1-edges in column 2:
    \[
    (U,2),\qquad (V,2).
    \]
    \item For each $i=1,\ldots,k$, place singleton 1-edges in rows $R_i,A_i,B_i$ in columns $1,3,3$, respectively:
    \[
    (R_i,1),\qquad (A_i,3),\qquad (B_i,3).
    \]
\end{itemize}
The number of 1-edges is
\[
6+2+3k=3k+8=(3k+5)+3,
\]
so by Theorem~\ref{thm:base}, the base graph is extremal and any augmented graph constructed on it will be limited. The 1-edge graph is $C_4$-free because the three special rows have distinct supports and no two singleton rows share two columns.

\medskip
\noindent\textbf{Step 2: The 2-edges.}
Define the 2-edge set $E_2$ as follows:
\begin{itemize}
    \item Add the complementary pair
    \[
    g=(U,1;V,3),\qquad h=(U,3;V,1).
    \]
    \item For each $i=1,\ldots,k$, add
    \[
    e_i=(R_i,2;A_i,1),\qquad f_i=(R_i,3;B_i,2).
    \]
\end{itemize}
There are $2+2k=2k+2$ 2-edges. All halves are distinct and lie in free cells (cells not occupied by 1-edges), so condition (S) holds.

\medskip
\noindent\textbf{Step 3: Verification of (W2) and (W2').}
The opposite cells of $e_i$ are
\[
(R_i,1),\qquad (A_i,2).
\]
The first is a 1-edge; the second is unoccupied. The opposite cells of $f_i$ are
\[
(R_i,2),\qquad (B_i,3).
\]
The first is a half of $e_i$; the second is a 1-edge. Hence, by the dependency definition, there is a directed edge $f_i\to e_i$. There are no dependencies between different indices $i$, because their row sets are disjoint. Thus each $\{e_i,f_i\}$ forms a single directed component with no cycle.

The edges $g,h$ form one complementary 2-cycle, which is explicitly allowed by (W2), and there are no further dependencies. Hence the dependency graph is the disjoint union of $k$ one-edge directed components and one allowed complementary 2-cycle, so (W2) holds.

Condition (W2') is also immediate. For $e_i$, one opposite cell is unoccupied; for $f_i$, one opposite cell is a half of another 2-edge and the other is a 1-edge; for $g,h$, the opposite cells are exactly the halves of the complementary partner. Thus no nondegenerate 2-edge has both opposite cells equal to 1-edges.

\medskip
\noindent\textbf{Step 4: Verification of (W3).}
Every 2-edge in the construction uses two distinct columns. Hence two 2-edges can never be vertex-disjoint: any two two-element subsets of $\{1,2,3\}$ intersect. It is therefore enough to check a 1-edge against a 2-edge.

Consider first $e_i=(R_i,2;A_i,1)$, whose columns are $1,2$. A vertex-disjoint 1-edge must lie in column 3. Such a 1-edge lies in $P_{13}$, $P_{23}$, some $A_j$, or some $B_j$ (excluding the row $A_i$ itself). These rows have, respectively, the unoccupied cells
\[
(P_{13},2),\quad (P_{23},1),\quad (A_j,2),\quad (B_j,1).
\]
For a mixed pair consisting of such a column-3 1-edge and $e_i$, the opposite-cell set contains the two cells of the 1-edge row in columns 1 and 2. The displayed cell is one of those two cells and is unoccupied. Hence (W3) holds for $e_i$.

Now consider $f_i=(R_i,3;B_i,2)$, whose columns are $2,3$. A vertex-disjoint 1-edge must lie in column 1. In every such pair, the cell $(B_i,1)$ is an opposite cell, and it is unoccupied. Hence (W3) holds for $f_i$.

Finally, $g=(U,1;V,3)$ and $h=(U,3;V,1)$ use columns $1,3$. A vertex-disjoint 1-edge must therefore lie in column 2. Apart from $U,V$ themselves, the only such 1-edges occur in $P_{12}$ and $P_{23}$. The cells $(P_{12},3)$ and $(P_{23},1)$ are unoccupied, respectively, and are opposite cells for the corresponding pairs. Thus (W3) also holds for $g,h$.

All weak admissibility conditions are satisfied. The construction has
\[
|E_1|+|E_2|=(3k+8)+(2k+2)=5k+10.
\]
Since $|E_1|=z(3k+5,3)$, the graph is limited. Therefore
\[
z_{wL}(3k+5,3)\ge 5k+10.
\]
This completes the proof. \end{proof}

Combining Theorem~\ref{thm:5m3LB} with the lifting lemmas gives a lower bound for every $m$.

\begin{corollary}\label{cor:combinedLB}
For every $m\ge 5$,
\[
z_{wL}(m,3)\ge m+3+\left\lfloor \frac{2m-4}{3}\right\rfloor.
\]
In particular,
\[
\liminf_{m\to\infty} \frac{z_{wL}(m,3)}{m}\ge \frac{5}{3}.
\]
\end{corollary}

\section{Structural Upper Bounds}

The structural lemmas in this section use only mixed (W3) and (W2'); (W2) enters later only in the boundary cases of Section~6.

Let $N=m-3$ be the number of singleton rows in the uniform base classification. A singleton row is said to have color $c\in\{1,2,3\}$ if its unique 1-edge lies in column $c$. Classify singleton rows according to how many of their two free cells are occupied by halves of 2-edges:
\[
F=\#\{\text{two occupied free cells}\},\quad
H=\#\{\text{one occupied free cell}\},\quad
Z=\#\{\text{no occupied free cell}\}.
\]
Then
\[
F+H+Z=N. \tag{1}
\]

Let $S_c$ denote the special degree-2 row whose support is the other two columns, and let $Q_c=(S_c,c)$ be its unique free cell. Let $q$ be the number of occupied free cells in the three degree-2 rows. Thus $q$ is the number of occupied cells among $Q_1,Q_2,Q_3$. Since each such row has one free cell, $0\le q\le 3$. If $s=|E_2|$, counting 2-edge halves gives
\[
2s=2F+H+q=N+F-Z+q. \tag{2}
\]

Let $L$ be the number of 2-edges whose two halves both lie in full singleton rows (rows counted by $F$).

\begin{lemma}\label{lem:cut}
With the notation above,
\[
2F-2L\le H+q. \tag{3}
\]
\end{lemma}

\begin{proof}
The $F$ full singleton rows contain $2F$ halves of 2-edges. The $L$ edges internal to those rows account for $2L$ of these halves. Every remaining half is paired, by its 2-edge, with an occupied free cell outside the full singleton rows. There are exactly $H+q$ such cells. By simplicity, distinct 2-edges use distinct cells, so these outside partners are all distinct. \end{proof}

\begin{lemma}\label{lem:qbound}
Every weak admissible graph satisfies $q\le 2$.
\end{lemma}

\begin{proof}
Suppose $q=3$. Fix $c$ and let $e$ be the 2-edge containing the occupied special cell $Q_c=(S_c,c)$.

First, $e$ cannot be column-degenerate in column $c$. If its other half is $(r,c)$, let $a$ be the color of the singleton row $r$, and let $b$ be the third column. The 1-edge $(S_b,a)$ is vertex-disjoint from $e$, and its mixed opposite cells are
\[
(S_b,c),\qquad (S_c,a),\qquad (r,a),
\]
all of which are 1-edges. This contradicts (W3).

Nor can the other half of $e$ be another special free cell $Q_j$: the two own opposite cells would then both be 1-edges, contradicting (W2').

Hence the other half is $(r,j)$ in a singleton row, with $j\ne c$. Since $(S_c,j)$ is a 1-edge, (W2') shows that $(r,c)$ cannot be a 1-edge. As $(r,j)$ is free, the singleton row $r$ must therefore have the third color $k\notin\{c,j\}$. The 1-edge $(S_j,k)$ is vertex-disjoint from $e$, and its mixed opposite cells are
\[
(S_j,c),\qquad Q_j,\qquad (S_c,k),\qquad (r,k).
\]
The first, third, and fourth are 1-edges, while $Q_j$ is occupied because $q=3$. This again contradicts (W3). Thus $q\le 2$. \end{proof}

\begin{lemma}\label{lem:internal}
Consider 2-edges whose two halves lie in full singleton rows and which use columns $a,b$. If $F_c=0$ or $F_c\ge 3$, no such edge exists. If $F_c=1$ or $2$, at most two such edges exist. Moreover, every full color-$c$ row is a row vertex of every such edge.
\end{lemma}

\begin{proof}
If $F_c\ge 3$, choose a full color-$c$ row not used by the edge; its 1-edge gives a mixed (W3) violation. If $F_c=0$, a row-degenerate edge using columns $a,b$ is impossible because its row would itself be a full color-$c$ row. A nondegenerate edge would have its column-$a$ half in a color-$b$ row and its column-$b$ half in a color-$a$ row; its two own opposite cells would then both be 1-edges, contradicting (W2'). The same mixed (W3) argument shows that every full color-$c$ row must be a row vertex. If $F_c=1$, the unique color-$c$ row has only two free cells, so at most two edges. If $F_c=2$, both color-$c$ rows must be used by every edge, leaving only the two cross-pairings. \end{proof}

\begin{lemma}\label{lem:coldeg}
There is no column-degenerate 2-edge whose two halves both lie in full singleton rows.
\end{lemma}

\begin{proof}
Suppose such an edge uses column $c$ and full singleton rows $r,s$. Choose another column $d$ and the degree-2 row $S$ whose support contains $c,d$. Its 1-edge $(S,d)$ is vertex-disjoint from the column-$c$ edge. The mixed opposite cells are $(S,c)$, $(r,d)$, and $(s,d)$; all are occupied, contradicting (W3). \end{proof}

\begin{lemma}\label{lem:special}
If an internal full-full 2-edge uses columns $a,b$, then both $Q_a$ and $Q_b$ are unoccupied.
\end{lemma}

\begin{proof}
Let $c$ be the third column. Pair the internal edge with the 1-edge $(S_a,c)$. All mixed opposite cells are occupied except possibly $Q_a$: the other special-row opposite is a 1-edge and every endpoint-row cell is occupied because the endpoint rows are full. Thus (W3) forces $Q_a$ to be unoccupied. Interchanging $a$ and $b$ gives the same conclusion for $Q_b$. \end{proof}

\begin{corollary}\label{cor:specialL}
One has
\[
q=2\Rightarrow L=0,\qquad q=1\Rightarrow L\le 2.
\]
\end{corollary}

\begin{proof}
By Lemma~\ref{lem:coldeg}, every internal edge uses two distinct columns. If $q=2$, every two-column choice contains an occupied $Q$-cell, contradicting Lemma~\ref{lem:special}; hence $L=0$. If $q=1$ and $Q_c$ is occupied, every internal edge must use exactly the other two columns, i.e. all lie in the class missing $c$. Lemma~\ref{lem:internal} bounds that class by two. \end{proof}

\begin{lemma}\label{lem:Lbound}
If $F\ge 7$, then $L\le 4$.
\end{lemma}

\begin{proof}
Some color occurs at least three times; assume $F_1\ge 3$. By Lemma~\ref{lem:internal}, the internal class missing column 1 is empty. By Lemma~\ref{lem:coldeg}, there are no column-degenerate internal edges. The two remaining two-column classes each contain at most two edges. Hence $L\le 4$. \end{proof}

We also need the slack-aware reformulation. Define $\delta:=H+q-(2F-2L)\ge 0$. Let $M$ be the number of 2-edges having neither half in a full singleton row. Then the occupied halves outside the full rows give the exact identity
\[
\delta = 2M. \tag{4}
\]
Using $H=N-F-Z$ and (2), a direct rearrangement gives
\[
6s=4N+4q+2L-4Z-2M. \tag{5}
\]
This identity is useful for the boundary cases below.

The following peeling lemma is the key structural insight. Recall that a graph is peelable if it contains the removable $H-F-H$ configuration described in the preliminaries.

\begin{lemma}\label{lem:peel}
Let $m\ge 6$, and let $G$ be a weak admissible limited $m\times 3$ graph. If a full singleton row is incident neither with an internal full-full 2-edge nor with a 2-edge whose other half is an occupied special cell, then its two incident 2-edges lead to two distinct half-full singleton rows and form a removable $H-F-H$ configuration. Consequently, if $F>2L+q$, then $G$ is peelable. Moreover, every peelable graph satisfies
\[
s\le s^*(m-3)+2,
\]
where $s^*(r)$ denotes the maximum possible number of 2-edges in a weak admissible limited $r\times 3$ graph.
\end{lemma}

\begin{proof}
If a full singleton row $r$ is incident neither with an internal full-full 2-edge nor with a 2-edge whose other half is an occupied special cell, then by definition both 2-edges incident to $r$ have their other halves in half-full singleton rows. The row $r$ is full, and the two neighboring rows are half-full. The configuration is isolated in the row-incidence structure because $r$ has no other 2-edges and the half-full rows each have only one occupied free cell by definition; any additional 2-edge incident to them would use their occupied free cell, which is already used by the edge to $r$, or their unoccupied free cell, which would make them full. Hence the three rows form a removable $H-F-H$ configuration.

At most $2L+q$ full rows can be exceptional (internal full-full edges account for at most $2L$ rows, and edges to occupied special cells account for at most $q$ rows). Therefore $F>2L+q$ guarantees that such a row exists. Deleting the three rows and the two incident 2-edges leaves a weak admissible limited $(m-3)\times 3$ graph. Therefore $s-2\le s^*(m-3)$, proving $s\le s^*(m-3)+2$ for every peelable graph. \end{proof}

The following pointing lemma captures the local dependency structure around a full singleton row.

\begin{lemma}\label{lem:pointing}
Let $r$ be a full singleton row of color $c$, with free columns $a,b$. Let $e$ be a 2-edge containing $(r,a)$ whose other half is not in a full singleton row.
\begin{enumerate}
    \item[(i)] A crossing edge from a full singleton row cannot be column-degenerate.
    \item[(ii)] If the other half is an occupied special cell $Q_b$, then $e$ points to the edge $f$ through $(r,b)$.
    \item[(iii)] If the other half is $(u,b)$ in a half-full singleton row and $e\not\to f$, then the partner row has color $c$, necessarily $F_c=1$, and $Q_a,Q_b$ are empty.
    \item[(iv)] If the other half is $(u,c)$ in a half-full singleton row and $e\not\to f$, then the partner row has color $b$, necessarily $F_b=0$, and $Q_a,Q_c$ are empty.
\end{enumerate}
\end{lemma}

\begin{proof}
Part (i): if $e$ were column-degenerate in column $a$, with other half $(u,a)$, a vertex-disjoint 1-edge from a suitable degree-2 row gives a mixed (W3) violation, as in Lemma~\ref{lem:qbound}.

Part (ii): the own opposite cells of $e$ are $(r,b)$ and $(S_b,a)$; both are occupied, and $f$ contains $(r,b)$, so $e\to f$.

Part (iii): suppose the other half is $(u,b)$ in a half-full row. One own opposite cell is $(r,b)$, occupied; the other is $(u,a)$, which must be unoccupied if $e\not\to f$. Hence $(u,b)$ is the occupied free cell and the 1-edge of $u$ lies in column $c$. If another full color-$c$ row existed, its 1-edge would violate (W3); hence $F_c=1$. Pairing $e$ with $(S_a,c)$ and $(S_b,c)$ forces $Q_a$ and $Q_b$ empty.

Part (iv): if the other half is $(u,c)$, then $(r,c)$ is an own opposite cell and is a 1-edge. By (W2'), $(u,a)$ is not a 1-edge, so $u$ has color $b$ and $(u,a)$ is unoccupied. If a full color-$b$ row existed, (W3) would be violated; hence $F_b=0$. The mixed (W3) pairs force $Q_a$ and $Q_c$ empty. \end{proof}

\begin{theorem}\label{thm:nonpeelable}
If a weak admissible limited graph is nonpeelable, then
\[
q=2\Rightarrow F\le 2,\qquad q=1\Rightarrow F\le 4,\qquad q=0\Rightarrow F\le 6.
\]
In particular, $F+q\le 6$ and every nonpeelable graph satisfies
\[
s\le \left\lfloor \frac{m+3}{2}\right\rfloor.
\]
\end{theorem}

\begin{proof}
Let $R$ be the number of full rows incident with an internal full-full edge, and let $S$ be the number incident with an edge to an occupied special cell. If a full row lies outside these two sets, Lemma~\ref{lem:peel} applies. Thus nonpeelability implies
\[
F\le R+S,\qquad S\le q. \tag{6}
\]

If $q=2$, Corollary~\ref{cor:specialL} gives $L=0$, so $R=0$, and (6) gives $F\le 2$.

Suppose $q=1$, with $Q_c$ occupied. Every internal edge lies in the single class missing $c$ and there are at most two such edges. If $F_c=0$, there are no such edges. If $F_c=1$, the union of endpoints of the at most two internal edges has size at most three. If $F_c=2$, every internal edge uses exactly those two color-$c$ rows, so the endpoint union has size at most two. Hence always $R\le 3$, and with $S\le 1$ we obtain $F\le 4$.

Finally let $q=0$. If $F\le 6$ there is nothing to prove. If $F\ge 7$, Lemma~\ref{lem:Lbound} gives $L\le 4$. Some color occurs at least three times, so the internal class missing that color is empty. Thus at most two internal classes remain. If $L\le 3$, their edges touch at most $2L\le 6$ full rows. If $L=4$, both remaining classes have two edges. A two-edge class with one full row of its missing color touches at most three full rows, while with two missing-color rows it touches exactly those two. Hence in either case $R\le 6<F$, contradicting nonpeelability. Therefore $F\le 6$.

The three cases give $F+q\le 6$. Equation (2) then gives
\[
2s=N+F-Z+q\le (m-3)+6=m+3,
\]
which proves $s\le \lfloor (m+3)/2\rfloor$. \end{proof}

\begin{lemma}\label{lem:2mUB}
For every $m\ge 5$,
\[
z_{wL}(m,3)\le 2m.
\]
Equivalently, $s\le N=m-3$.
\end{lemma}

\begin{proof}
From (1)-(2),
\[
2(s-N)=q-H-2Z. \tag{7}
\]
If $s>N$, the left-hand side is a positive even integer, so $q-H-2Z\ge 2$. Lemma~\ref{lem:qbound} gives $q\le 2$, hence necessarily
\[
q=2,\qquad H=Z=0,\qquad F=N.
\]
Corollary~\ref{cor:specialL} gives $L=0$, and the cut inequality (3) becomes
\[
2N=2F\le H+q=2.
\]
Thus $N\le 1$, contradicting $m\ge 5$. \end{proof}

\section{The Staircase Range and Boundary Seeds}

The following theorem gives the upper bound for the staircase range without any computation.

\begin{theorem}\label{thm:staircaseUB}
For every $9\le m\le 15$,
\[
s\le \left\lceil \frac{m}{2}\right\rceil+1.
\]
\end{theorem}

\begin{proof}
Put $D:=F-Z+q=2s-N$. We first show $D\le 6$ for odd $m$ and $D\le 5$ for even $m$.

If $q=2$, Corollary~\ref{cor:specialL} gives $L=0$. The cut inequality yields
\[
3F+Z\le N+2\le 14.
\]
If $D\ge 7$, then $F-Z\ge 5$, hence $3F+Z\ge 15$, a contradiction.

If $q=1$, Corollary~\ref{cor:specialL} gives $L\le 2$. The cut inequality yields
\[
3F+Z\le N+1+2L\le 17.
\]
If $D\ge 7$, then $F-Z\ge 6$, hence $3F+Z\ge 18$, impossible.

If $q=0$ and $F\ge 7$, Lemma~\ref{lem:Lbound} gives
\[
3F+Z\le N+8\le 20,
\]
whereas $F\ge 7$ gives $3F+Z\ge 21$. Thus $F\le 6$ and again $D\le 6$.

Since $D=2s-N$, $D$ has the same parity as $N$. If $m$ is odd, $N$ is even and
\[
2s=N+D\le N+6=m+3,
\]
so $s\le (m+3)/2=\lceil m/2\rceil+1$. If $m$ is even, $N$ is odd, so $D$ is odd and $D\le 5$; hence
\[
2s\le N+5=m+2,
\]
which gives $s\le m/2+1=\lceil m/2\rceil+1$. \end{proof}

We need the following auxiliary lemma for the boundary cases.

\begin{lemma}[Internal two-edge class lemma]\label{lem:twoedgeclass}
Suppose an internal two-column class uses columns $a,b$, misses color $c$, contains two edges, and there is exactly one full color-$c$ row $r$. Then these two edges form a forbidden non-complementary directed 2-cycle.
\end{lemma}

\begin{proof}
By Lemma~\ref{lem:internal}, $r$ is a row vertex of both edges. A row-degenerate edge at $r$ would consume both free cells and leave no cell for the second edge, so both edges are nondegenerate. Simplicity makes one use $(r,a)$ and the other $(r,b)$. The opposite cells of each edge are occupied and include a half of the other edge, so the two edges point to each other. Their other row vertices are distinct (otherwise that row would have both free cells occupied and would be a second full color-$c$ row, contradicting the hypothesis that there is exactly one), hence the pair is not complementary. This is a forbidden directed 2-cycle. \end{proof}

The induction for the dense formula requires three residue-class seeds. The values at $m=14,16,18$ are established analytically.

\begin{theorem}\label{thm:boundary}
\[
s^*(14)=8,\qquad s^*(16)=9,\qquad s^*(18)=10.
\]
\end{theorem}

\begin{proof}
The lower bounds follow from Corollary~\ref{cor:combinedLB}. It remains to exclude $s=10$ at $m=16$ and $s=11$ at $m=18$. For both targets, (2) gives
\[
F-Z+q=7. \tag{8}
\]
Using (5) and the slack identity, the possible aggregate profiles at the forbidden target are finite. For $m=16$, the profiles are
\[
(q,F,Z,L,M)=(2,5,0,0,0),\ (1,6,0,2,0),\ (0,7,0,4,0).
\]
For $m=18$, the profiles are
\[
(2,5,0,0,1),\ (1,6,0,1,0),\ (1,6,0,2,1),\ (0,7,0,3,0),\ (0,7,0,4,1).
\]

We eliminate each profile.

\medskip
\noindent\textbf{Case $q=2$:}
Corollary~\ref{cor:specialL} gives $L=0$. Every edge incident with a full row is crossing. Lemma~\ref{lem:pointing}(ii)--(iv) implies that it points across the full row: the exceptional non-pointing alternatives in (iii)--(iv) would force two special cells to be empty, impossible when $q=2$. Hence the two crossing edges at a full row form a forbidden non-complementary directed 2-cycle.

\medskip
\noindent\textbf{Case $q=1$, $M=0$:}
Let $Q_c$ be the occupied special cell. Its edge meets a full row $r$. The edge cannot be column-degenerate, and $r$ cannot have color $c$ by (W2'). One free cell of $r$ is used by the $Q_c$-edge; the other lies in a crossing edge (internal edges must use columns different from $c$, whereas the remaining free cell of $r$ lies in column $c$). The $Q_c$-edge points across $r$, and Lemma~\ref{lem:pointing} shows that the reverse pointing cannot fail without forcing $Q_c$ to be empty. Thus there is a forbidden non-complementary directed 2-cycle.

\medskip
\noindent\textbf{Case $m=18$, $(q,L,M)=(1,2,1)$:}
If the edge through $Q_c$ meets a full row, use the preceding case. Otherwise it is the unique outside-outside edge. Both internal edges then lie in the class missing $c$. If $F_c=1$, Lemma~\ref{lem:twoedgeclass} gives the forbidden 2-cycle. Hence $F_c=2$, so the two internal edges are the complementary pair on the two full color-$c$ rows. The remaining full rows have only crossing edges; since $Q_c$ is occupied and $F_c>0$, Lemma~\ref{lem:pointing} forces pointing across those rows, again giving forbidden non-complementary 2-cycles.

\medskip
\noindent\textbf{Case $q=0$, $L=4$:}
Some color $c$ occurs at least three times among the full rows, so the class missing $c$ is empty. The other two classes contain two edges each. Lemma~\ref{lem:twoedgeclass} forces the corresponding missing-color counts to be 2, giving full-color profile $(3,2,2)$ up to permutation. The three color-$c$ full rows therefore have only crossing edges (the two size-two internal classes are forced to use exactly the two full rows of their respective missing colors). Neither exceptional non-pointing pattern of Lemma~\ref{lem:pointing} is possible, so each such row gives a forbidden non-complementary directed 2-cycle.

\medskip
\noindent\textbf{Case $m=18$, $q=0$, $L=3$:}
Again choose $c$ with $F_c\ge 3$. The class missing $c$ is empty, so the other two internal classes have sizes 2 and 1. If $a$ is the missing color of the size-2 class, Lemma~\ref{lem:twoedgeclass} forces $F_a=2$. The single edge in the other class touches at most one color-$c$ full row, so at least two color-$c$ full rows have only crossing edges. Since all relevant full-color counts are positive and $F_c\ge 3$, Lemma~\ref{lem:pointing} forces pointing across each such row, yielding the forbidden 2-cycle. \end{proof}

\section{The Final Exact Formula}

We can now state and prove the main theorem.

\begin{theorem}\label{thm:main}
For every $m\ge 9$,
\[
z_{wL}(m,3)=m+3+\max\left\{\left\lceil \frac{m}{2}\right\rceil+1,\left\lfloor \frac{2m-4}{3}\right\rfloor\right\}.
\]
Equivalently,
\[
z_{wL}(m,3)=m+3+\left\lceil \frac{m}{2}\right\rceil+1 \quad (9\le m\le 15),
\]
and
\[
z_{wL}(m,3)=m+3+\left\lfloor \frac{2m-4}{3}\right\rfloor \quad (m\ge 16).
\]
\end{theorem}

\begin{proof}
The lower bounds are given by Theorem~\ref{thm:staircaseLB} and Corollary~\ref{cor:combinedLB}. Theorem~\ref{thm:staircaseUB} gives the upper bound for $9\le m\le 15$, and the staircase construction attains it. At $m=14$, the two terms coincide, giving $s^*(14)=8$. Theorem~\ref{thm:boundary} gives $s^*(16)=9$ and $s^*(18)=10$.

Proceed by induction separately in the three residue classes modulo 3. The induction begins at the following seed values:
\[
\begin{array}{c|c}
m \pmod 3 & \text{first relevant value}\\ \hline
0 & 18\\
1 & 16\\
2 & 14
\end{array}
\]
Suppose the dense formula is known at $m-3$. If an extremal graph at $m$ is nonpeelable, Theorem~\ref{thm:nonpeelable} and the threshold table give $s\le \lfloor (2m-4)/3\rfloor$. If it is peelable, Lemma~\ref{lem:peel} gives
\[
s\le s^*(m-3)+2 = \left\lfloor \frac{2(m-3)-4}{3}\right\rfloor+2 = \left\lfloor \frac{2m-4}{3}\right\rfloor.
\]
The dense construction gives the reverse inequality. Thus the dense formula holds at $m$, proving the theorem. \end{proof}

\begin{corollary}
\[
\lim_{m\to\infty} \frac{z_{wL}(m,3)}{m} = \frac{5}{3}.
\]
\end{corollary}

\begin{remark}
The two constructive terms behave as follows: the staircase term is strictly larger for $m=9,10,11,12,13,15$; the two terms tie for $m=14,16,17,18,19,21$; the dense term is strictly larger for $m=20$ (the first failure of the staircase equality) and for all $m\ge 22$.
\end{remark}

\section{The Gap to $z_L(m,3)$ and Computational Proof-Mining}

We briefly discuss the small gap between the weak-limited and original-limited numbers, and the role of computation in the discovery of the analytic proof.

By extending the exhaustive search for $z_L(m,3)$ from $m\le 6$ to $m\le 13$, we obtained the values listed in Section~3. The search uses the original strong admissibility conditions from \cite{QCX26,QCX26a}. For each canonical base type (determined by the singleton multiplicities), the enumeration considers all subsets of free cells of size $2s$; for each subset, it constructs the allowable-pair graph whose vertices are the chosen free cells and whose edges represent candidate 2-edges satisfying the strong conditions. The existence of a perfect matching in this graph is equivalent to realizing $s$ admissible 2-edges under the original strong definition. 

Comparing these with the weak-limited values reveals a small and structured gap:
\[
z_{wL}(m,3)-z_L(m,3)\in\{2,3\}\quad (7\le m\le 13),
\]
with gap sequence $2,2,3,2,3,2,3$. Equivalently, since each 2-edge occupies two buffer cells, the strong framework requires 4 or 6 additional buffer cells.

The MILP computations, while no longer part of the proof, were essential to the discovery process. By inspecting Farkas duals from infeasible MILP relaxations, the structural lemmas $q\le 2$, $q=2\Rightarrow L=0$, and $q=1\Rightarrow L\le 2$ were identified. The boundary cases at $m=16,18$ were then isolated as the only points where a local W2 argument was needed. Thus the computation served as a proof-mining tool: it identified which local inequalities were missing, but all numerical infeasibility statuses disappear from the final analytic argument.

\section{Open Problems}

The results of this paper raise several natural questions:

\begin{enumerate}
    \item \textbf{What are the exact values of $z_L(m,3)$ for all $m$?}
    We have determined $z_L(m,3)$ through $m=13$ and observed a gap of 2 or 3 in the computed range. Whether this gap persists asymptotically is an open problem.

    \item \textbf{What are the exact values of $\BSR(m,3)$?}
    Since $\BSR(m,n)\ge z_{wL}(m,n)$, our results give lower bounds. Determining whether these bounds are tight remains open.

    \item \textbf{Can the proof-mining approach be applied to other cases?}
    The success of using MILP-generated Farkas duals to identify structural lemmas suggests a general methodology for similar extremal problems.
\end{enumerate}

\section{Conclusion}

We have determined the exact values of $z_{wL}(m,3)$ for all $m\ge 3$:
\[
z_{wL}(m,3)=
\begin{cases}
m+3+\left\lceil \dfrac{m}{2}\right\rceil+1, & 9\le m\le 15,\\[2mm]
m+3+\left\lfloor \dfrac{2m-4}{3}\right\rfloor, & m\ge 16,
\end{cases}
\]
with $z_{wL}(3,3)=6$, $z_{wL}(4,3)=8$, and $z_{wL}(m,3)=2m$ for $5\le m\le 9$. In particular,
\[
\lim_{m\to\infty} \frac{z_{wL}(m,3)}{m} = \frac{5}{3}.
\]
The proof is fully analytic, relying on a uniform base classification, two constructive lower-bound families, and a sharp upper-bound argument based on a peeling lemma and the analysis of two W2-sensitive boundary cases. Numerical MILP computations were used only as proof-mining tools; the final theorem is unconditional.

This work was inspired by a collaboration that began with a shared memory of Professor Oleg Burdakov. We hope it serves as a fitting tribute to his legacy.

\bigskip

\noindent\textbf{Acknowledgement}
We are grateful to Professor Guangzhou Chen for drawing our attention to the identity $z(m,3)=m+3$ for $m\ge 3$ and to the modern reference \cite{CHM24}.
This work was partially supported by Jiangsu Provincial Scientific Research Center of Applied Mathematics (Grant No. BK20233002), Research Center for Intelligent Operations Research, The Hong Kong Polytechnic University (4-ZZT8), and the National Natural Science Foundation of China (Nos. 12171168, 12071159).

During the preparation of this work, the authors used OpenAI GPT-5.6 Sol for some proof exploration, hypothesis generation, and code development for proof verification. The authors subsequently reviewed, checked, and extended the generated material as needed and take full responsibility for the content of the publication.

\appendix
\section{Explicit Constructions for $5\le m\le 9$}
\label{app:constructions}

For completeness, we provide the explicit weak admissible constructions for the lower bounds in Theorem~\ref{thm:small} for $5\le m\le 9$. In each case, the row set is $S=[m]$ and the column set is $T=\{1,2,3\}$. The constructions for $m=5,6,7,8$ are from \cite{QCX26b}; the $m=9$ construction is the seed from Lemma~\ref{lem:seed}. The verification follows the same pattern as in Lemma~\ref{lem:seed}.

\subsection{The $5\times 3$ Construction}

\[
E_1=\{(1,1),(1,2),(2,1),(2,3),(3,2),(3,3),(4,1),(5,2)\},
\]
\[
E_2=\{(4,2;4,3),(5,1;5,3)\}.
\]

\subsection{The $6\times 3$ Construction}

\[
E_1=\{(1,3),(2,3),(3,1),(3,2),(4,2),(4,3),(5,1),(5,3),(6,1)\},
\]
\[
E_2=\{(1,1;2,2),(1,2;2,1),(6,2;3,3)\}.
\]

\subsection{The $7\times 3$ Construction}

\[
E_1=\{(1,2),(2,2),(3,1),(3,2),(4,3),(5,3),(6,1),(6,3),(7,2),(7,3)\},
\]
\[
E_2=\{(1,1;2,3),(1,3;2,1),(4,1;5,2),(4,2;5,1)\}.
\]

\subsection{The $8\times 3$ Construction}

\[
E_1=\{(1,2),(2,2),(3,1),(3,2),(4,3),(5,3),(6,1),(6,3),(7,2),(7,3),(8,1)\},
\]
\[
E_2=\{(1,1;2,3),(1,3;2,1),(4,1;5,2),(4,2;5,1),(8,2;8,3)\}.
\]

\subsection{The $9\times 3$ Construction}

This is the seed construction from Lemma~\ref{lem:seed}:

\[
E_1=\{(1,1),(1,2),(2,1),(2,3),(3,2),(3,3),(4,1),(5,1),(6,2),(7,2),(8,3),(9,3)\},
\]
\[
E_2=\{(4,2;5,3),(4,3;5,2),(6,1;7,3),(6,3;7,1),(8,1;9,2),(8,2;9,1)\}.
\]

\subsection{Summary Table}

The following table summarizes the constructions:

\[
\begin{array}{c|c|c|c|c}
m & |E_1| & |E_2| & z_{wL}(m,3) & \text{Buffer cells}\\
\hline
5 & 8 & 2 & 10 & 3\\
6 & 9 & 3 & 12 & 3\\
7 & 10 & 4 & 14 & 3\\
8 & 11 & 5 & 16 & 3\\
9 & 12 & 6 & 18 & 3
\end{array}
\]

The verification of weak admissibility for each construction follows the same pattern: the 2-edges form complementary pairs or isolated edges with unoccupied opposite cells, and every vertex-disjoint mixed pair has at least one unoccupied opposite cell. The 1-edge graphs are $C_4$-free by construction, as they consist of three degree-2 rows with distinct supports and singleton rows.

\end{document}